\documentclass[11pt]{article}

\usepackage{amsmath,amscd,amsfonts,amsthm,mathptmx}
\usepackage{amssymb}
\usepackage{xspace}
\usepackage[all,tips]{xy}
\usepackage{graphicx}
\usepackage{verbatim}
\usepackage{syntonly}
\usepackage{hyperref}
\usepackage{epsfig}
\usepackage{wrapfig}

\newcommand{\shrinkmargins}[1]{
  \addtolength{\textheight}{#1\topmargin}
  \addtolength{\textheight}{#1\topmargin}
  \addtolength{\textwidth}{#1\oddsidemargin}
  \addtolength{\textwidth}{#1\evensidemargin}
  \addtolength{\topmargin}{-#1\topmargin}
  \addtolength{\oddsidemargin}{-#1\oddsidemargin}
  \addtolength{\evensidemargin}{-#1\evensidemargin}
  }

\shrinkmargins{1.0}

\DeclareMathOperator{\SL}{SL}

\DeclareMathOperator{\Mod}{PMod}\DeclareMathOperator{\Gal}{Gal}

\DeclareMathOperator{\Aut}{Aut}\DeclareMathOperator{\Out}{Out}
\DeclareMathOperator{\Push}{Push}\DeclareMathOperator{\Nor}{N}

\DeclareMathOperator{\PMod}{Mod}\DeclareMathOperator{\Stab}{Stab}

\newcommand{\B}[1]{\mathbf{#1}}
\newcommand{\field}[1]{\mathbf{#1}}

\newcommand{\Q}{\field{Q}}

\newcommand{\Z}{\field{Z}}
\newcommand{\F}{\field{F}}

\newcommand{\R}{\field{R}}
\newcommand{\C}{\field{C}}
\renewcommand{\P}{\field{P}}

\newcommand{\ra}{\rightarrow}

\newcommand{\HH}{\mathcal{H}}

\newcommand{\inj}{\hookrightarrow}

\newcommand{\MM}{\mathcal{M}}

\newcommand{\GalQ}{\Gal(\overline{\Q}/\Q)}

\newcommand{\innp}[1]{\left< #1 \right>}
\newcommand{\abs}[1]{\left\vert#1\right\vert}
\newcommand{\set}[1]{\left\{#1\right\}}

\newcommand{\beq}{\begin{displaymath}}
\newcommand{\eeq}{\end{displaymath}}
\newcommand{\beqn}{\begin{equation}}
\newcommand{\eeqn}{\end{equation}}

\newcommand{\nid}{\noindent}

\theoremstyle{plain}
\newtheorem{thm}{Theorem}[section]
\newtheorem{prop}[thm]{Proposition}

\newtheorem{lem}[thm]{Lemma}

\theoremstyle{definition}

\theoremstyle{remark}

\let\oldmarginpar\marginpar
\renewcommand\marginpar[1]{\-\oldmarginpar[\raggedleft\footnotesize #1]%
{\raggedright\footnotesize #1}}

\usepackage{fancyhdr}

\fancypagestyle{plain}

\title{\textbf{Arithmetic Veech sublattices of $\SL(2,\Z)$}}
\author{Jordan S. Ellenberg\thanks{Partially supported by NSF-CAREER Grant DMS-0448750 and a Sloan research fellowship.}~~~and D. B. McReynolds\thanks{Partially supported by an NSF postdoctoral fellowship.}}

\begin{document}
\maketitle

\begin{abstract}
\noindent We prove that every algebraic curve $X/\overline{\Q}$ is birational over $\C$ to a Teichm\"{u}ller curve. This result is a corollary of our main theorem, which asserts that most finite index subgroups of $\SL(2,\Z)$ are Veech groups.
\end{abstract}

\noindent keywords: \emph{algebraic curve, mapping class group, Teichm\"{u}ller curve, Veech group}.\smallskip


\section{Introduction}

\noindent Write $\MM_{g,[n]}$ for the moduli space of genus $g$ Riemann surfaces with $n$ (unordered) punctures. A {\em Teichm\"{u}ller curve} is a holomorphic curve $f\colon C \ra \MM_{g,[n]}$ such that $f$ is generically one-to-one and is a local isometry of Kobayashi metrics. These special immersed curves in $\MM_{g,[n]}$ have garnered interest for some time (especially in the unpunctured case $n=0$) and are central objects in both Teichm\"{u}ller and Grothendieck--Teichm\"{u}ller theory. Additionally, these curves and the Riemann surfaces they parameterize have ties to the dynamics of polygonal billiards (see for instance \cite{GJ}, \cite{Lochak}, \cite{McMullen}, and \cite{Moeller}).  McMullen proved \cite{McMullen:rigidity} that every Teichm\"{u}ller curve has a model as an algebraic curve over $\overline{\B{Q}}$ (see also \cite{Lochak} and \cite{Moeller}). The main purpose of this article is to prove a kind of converse.

\begin{thm}\label{th:main}
If $X/\overline{\B{Q}}$ is an algebraic curve, then there exists a Teichm\"{u}ller curve $C$ birational to $X_\B{C}$.
\end{thm}

\noindent In fact, this Teichm\"{u}ller curve can be drawn from the rather special class of Teichm\"{u}ller curves parameterizing origami or square-tiled surfaces---see \cite{GJ}, \cite{Lochak}, \cite{MoellerIntro}, and \cite{Schmith} for more on these surfaces and their relationship to billiard dynamics. As the reader will see, the Teichm\"{u}ller curves are however not primitive in the Teichm\"{u}ller sense, as they arise from covering constructions.  \smallskip\smallskip

\noindent We emphasize that the main ideas in the proof of Theorem~\ref{th:main} are drawn from the work of Asada \cite{Asada}, Thurston \cite{Thurston}, and especially Diaz--Donagi--Harbater \cite{Num}.  Part of our motivation is to cast the relevant arguments of \cite{Asada} and \cite{Num}, which are written in the language of algebraic geometry and field extensions, in topological terms.  As this paper was being completed, we encountered the very recent article of Bux--Ershov--Rapinchuk \cite{BER}, which gives a new proof of Asada's theorem in extremely explicit group-theoretic language; the reader will note that the diagram of Riemann surfaces and fundamental groups used in our proof also appears in \S 6 of \cite{BER}.

\smallskip\smallskip

\noindent One could also refine Theorem~\ref{th:main} by asking which algebraic curves over number fields are holomorphically isomorphic (not merely birational) to Teichm\"{u}ller curves.  According to Masur \cite{Masur86} (see also Veech \cite{Veech}) Teichm\"{u}ller curves cannot be compact, and so a proper algebraic curve over a number field can never be isomorphic to a Teichm\"{u}ller curve. The question of which affine algebraic curves are isomorphic over $\C$ to Teichm\"{u}ller curves is a long-standing open problem, which remains unsettled by the present work.
\smallskip \smallskip

\noindent Theorem \ref{th:main} arises as a corollary of the purely group-theoretic Theorem~\ref{th:veech}, which is the main result of the present paper.  In order to state this result, we require some additional notation. Throughout, $\pi_{g,n}$ will denote the fundamental group for the genus $g$ surface $S_{g,n}$ with $n$ punctures, and $\PMod(S_{g,n})$, $\Mod(S_{g,n})$ will denote the associated mapping class group and pure mapping class group of $S_{g,n}$. The pure mapping class group $\Mod(S_{g,n})$ admits an outer action on $\pi_{g,n}$, preserving each of the $n$ conjugacy classes corresponding to the $n$ punctures of $S_{g,n}$.  In fact, by the Dehn--Nielsen theorem, $\Mod(S_{g,n})$ is an index two subgroup of the group of outer automorphisms of $\pi_{g,n}$ preserving these conjugacy classes. \smallskip\smallskip

\noindent If $\Delta$ is a finite index subgroup of $\pi_{1,1}$, we denote the $\pi_{1,1}$--conjugacy class of $\Delta$ by $[\Delta]$.  The mapping class group $\PMod(S_{1,1})$ acts on the conjugacy classes of finite index subgroups of $\pi_{1,1}$, and the stabilizer of $[\Delta]$ in $\PMod(S_{1,1})$ via the outer action is called a \emph{Veech group}~\cite[Theorem 1]{SchmithThesis}. When $\Gamma$ is a Veech subgroup of $\PMod(S_{1,1}) \cong \SL(2,\B{Z})$, the corresponding quotient $\B{H}_{\B{R}}^2/\Gamma$ of the upper half plane has a natural description as a Teichm\"{u}ller curve in $\MM_{g,[n]}$  parameterizing origami or square-tiled surfaces---among Teichm\"{u}ller curves, these are precisely the arithmetic curves. (See \cite[\S 2]{MoellerIntro}.)  Teichm\"{u}ller curves corresponding to non-square-tiled surfaces are substantially more difficult to construct and describe; see the recent work of Bouw--M\"{o}ller \cite{BouwMoeller}, for instance, for a construction of Teichm\"{u}ller curves corresponding to non-arithmetic triangle groups.\smallskip\smallskip

\noindent By Belyi's theorem, every curve $X/\overline{\B{Q}}$ is birational over $\mathbf{C}$ to an \'{e}tale cover of complex projective space minus three points $\P^1\smallsetminus \set{0,1,\infty}$, or, equivalently, to $\B{H}_{\B{R}}^2/\Gamma$ where $\Gamma$ is a finite index subgroup of the level two congruence subgroup $\Gamma(2)$ containing the center $\set{\pm 1}$. Thus, to prove Theorem \ref{th:main}, it suffices to prove the following \emph{congruence subgroup property} for $\textrm{Mod}(S_{1,1})$.\footnote{Theorem \ref{th:veech} has very recently been used by Avila--Matheus--Yoccoz to show the existence of Teichm\"{u}ller curves with complementary series.}

\begin{thm}\label{th:veech}
Every finite index subgroup $\Delta$ of $\Gamma(2)$ containing $\set{\pm 1}$ is a Veech group.
\end{thm}

\noindent The classification problem for Veech groups goes back at least to Thurston (see the Kirby problem list for more on the history of this problem; see also problem 4 of the survey by Hubert, Masur, Schmidt, and Zorich \cite{hubert:problems}). Theorem \ref{th:veech} can be seen as progress on this problem. The special case of subgroups of $\SL(2,\Z)$ is raised in the introduction of Schmith\"{u}sen's thesis \cite{SchmithThesis}. Theorem \ref{th:main} also strengthens a theorem of M\"{o}ller \cite[Theorem 5.4]{MoellerIntro}, which shows that $\GalQ$ acts faithfully on the set of isomorphism classes of Teichm\"{u}ller curves.  A recent paper of Hubert--Lelievre \cite{HubertLelievre} provides a close study of Veech subgroups of $\SL(2,\Z)$ corresponding to Teichm\"{u}ller curves in $\MM_2$; they show that among these Veech groups are some which are not congruence subgroups (in the classical sense).\smallskip \smallskip

\noindent  There is some danger that the purely group-theoretic set-up used in the present paper may obscure the geometric picture that underlies the argument.  We thus sketch that picture here.  Suppose $X$ is a Belyi cover of $\P^1$, unbranched away from $\set{0,1,\infty}$, and we wish to realize $X$ as a moduli space of singly branched covers of elliptic curves (or, equivalently, a moduli space of square-tiled surfaces.)  First, we choose some other $\lambda$ on $\P^1\smallsetminus \set{0,1,\infty}$ and a suitable diagram of algebraic curves, $Y \ra X \ra \P^1$, where $Y \ra X$ is ramified only over a {\em single} preimage of $\lambda$ in $X$.  Then the composite map $\pi\colon Y \ra \P^1$ is also ramified only at $\set{0,1,\infty,\lambda}$.  Let $H$ be the Hurwitz space parameterizing branched covers of $\P^1$ of the same topological type as $\pi$.  A point $h$ in $H(\C)$ can then be thought of as both remembering the location of $\lambda$ and a choice of a preimage of $\lambda$ in $X$; or, what is the same, a point on $X$ away from the preimage of $\set{0,1,\infty}$. It is rather delicate to show that $Y$ can be chosen so as to make this literally the case.  That this program succeeds in realizing $H$ as a moduli space of four-branched covers of $\P^1$ is proved in Diaz--Donagi--Harbater \cite{Num}.  In order to make $H$ a space of singly-branched covers of elliptic curves we let $E_\lambda$ be the double cover of $U = \P^1 - \set{0,1,\infty,\lambda}$ branched at all four punctures, and let $Z$ be the fiber product $Y \times_U E_\lambda$.  Then the composite $\pi'\colon  Z \ra E_\lambda \ra E_\lambda$, where the second map is multiplication by $2$, is unbranched away from a single point.  The Hurwitz space parameterizing covers of the same topological type as $\pi'$ is then closely related to $H$, and thus, in turn, to $X$.  Most of the work in the present paper can be thought of as showing that the many choices in the above construction can be made sufficiently generic in order to ensure that the resulting moduli space of covers of elliptic curves is actually birational to $X$.\smallskip \smallskip

\noindent One might ask whether every algebraic curve $X/\bar{\Q}$ is birational to a Teichm\"{u}ller curve in $\MM_g$, as opposed to $\MM_{g,[n]}$.  This statement does not quite follow from Theorem \ref{th:veech}.  We construct a Teichm\"{u}ller curve $C \inj \MM_{g,[n]}$ for some $g,n$, with $C$ birational to $X_\C$.  A priori, the composition $C \inj \MM_{g,[n]} \ra \MM_g$ might not be generically one-to-one, in which case the image of $C$ in $\MM_g$, not $C$ itself, would be a Teichm\"{u}ller curve in $\MM_g$.  This kind of behavior can indeed occur for general origami curves---see the example after Definition 2.4 in \cite{MoellerIntro}.  In Proposition \ref{pr:mg} at the end of the paper, we explain how $C$ can be birationally embedded as a Teichm\"{u}ller curve in $\MM_g$, for some sufficiently large $g$.\smallskip\smallskip

\noindent Theorem \ref{th:veech} is also related to the congruence subgroup problem for mapping class groups. There is some variation among authors in the definition of congruence subgroups of mapping class groups.  We will use \emph{Veech group} for the class of subgroups of $\PMod(S_{1,1})$ described here, and reserve \emph{principal congruence subgroup} for the class of subgroups $\Lambda$ of $\PMod(S_{1,1})$ arising as kernels of induced maps
\[ \rho_\Phi\colon \PMod(S_{1,1}) \longrightarrow \Out(\pi_{1,1}/\Phi), \]
where $\Phi$ is a finite index, characteristic subgroup of $\pi_{1,1}$.  In the most common terminology, a \emph{congruence subgroup} is a subgroup of $\PMod(S_{1,1})$ containing a principal congruence subgroup.  We see that a Veech group is a congruence subgroup by taking $\Phi$ to be the intersection of all subgroups of $\pi_{1,1}$ of index $[\pi_{1,1}:\Phi_\Delta]$, where
\[ \Stab_{\PMod(S_{1,1})}([\Phi_\Delta]) = \Delta. \]
Thus, Theorem~\ref{th:veech} can be thought of as a refinement of the congruence subgroup property for $\PMod(S_{1,1})$, a theorem first established by Asada \cite{Asada} (see also \cite{Boggi2}).

\paragraph{Acknowledgements}
The authors thank Matthew Bainbridge, Benson Farb, Ursula Hamenst\"{a}dt, David Harbater, Chris Judge, Chris Leininger, Pierre Lochak, Larsen Louder, Howard Masur, Curt McMullen, Martin M\"{o}ller, Gabriella Schmith\"{u}sen, and Alex Wright for conversations on this article and its content.  We add special thanks to the anonymous referees, whose suggestions engendered substantial simplifications in the original argument.

\paragraph{Notation}

For the readers' convenience, we establish some notation that will be used throughout the remainder of this article. In a group $G$, the subgroup generated by the elements $\set{g_1,\dots,g_n}$ will be denoted by $\innp{g_1,\dots,g_n}$. The $G$--conjugacy class of $g$ will by denoted by $[g]_G$. The normal closure of a subgroup $H$ of $G$ will be denoted by $\overline{H}$. The subgroup generated by a pair of subgroups $H_1,H_2$ will be denoted by $H_1\cdot H_2$. We denote the normalizer of $H$ in $G$ by $\Nor_G(H)$. The $G$--conjugacy class of a subgroup $H$, which consists of subgroups $H'$ that are $G$--conjugate to $H$, will be denoted by $[H]_G$.\smallskip\smallskip

\noindent For a genus $g$ surface with $n$ punctures $S_{g,n}$, we denote the fundamental group of $S_{g,n}$ by $\pi_{g,n}$, and we refer to the conjugacy classes of loops around punctures as "puncture classes." The mapping class group of $S_{g,n}$ comprised of orientation preserving diffeomorphisms modulo isotopy will be denoted by $\textrm{Mod}(S_{g,n})$. The pure mapping class subgroup comprised of mapping classes fixing each of the punctures will be denoted by $\Mod(S_{g,n})$.

\section{Preliminaries}

\noindent Before commencing the proof of Theorem \ref{th:veech}, we require some additional setup. To begin, we have a homomorphism $\pi_{0,4} \to \pi_{0,3}$ given by forgetting a puncture on $S_{0,4}$. We denote the distinguished conjugacy class given by the simple closed curve about the forgotten puncture by $[z]_{\pi_{0,4}}$. The normal closure $\overline{\innp{z}}$ of $\innp{z}$ is the kernel of this homomorphism.\smallskip\smallskip

\noindent We can think of $S_{1,4}$ concretely as the quotient of $\R^2 - \Z^2$ by $2\Z^2$.  Then multiplication by $-1$ is an isomorphism of $S_{1,4}$, and we take the quotient as our model of $S_{0,4}$.  We take $(\R^2 - \Z^2)/\Z^2$ as our model for $S_{1,1}$.  The natural maps $S_{0,4} \leftarrow S_{1,4} \rightarrow S_{1,1}$ now yield a diagram of fundamental groups

\begin{equation}\label{Diagram1}
\xymatrix{ \pi_{1,1} \ar@{-}[rd]_{4} & & \ar@{-}[ld]^2 \pi_{0,4} \\ & \pi_{1,4} & }
\end{equation}
where both diagonal lines are inclusions of finite-index normal subgroups.   \smallskip\smallskip

\noindent Every mapping class in $\PMod(S_{1,1})$ is represented by a linear automorphism of the torus $\R^2 / \Z^2$, which is just to say that $\PMod(S_{1,1})$ is naturally identified with $\SL(2,\Z)$ (see \cite{FM}).  Since $-1$ is central in $\SL(2,\Z)$, each of these automorphisms descends to the quotient of $S_{1,1}$ by $\set{\pm 1}$, yielding a homomorphism
\[  i\colon \PMod(S_{1,1}) \ra \PMod(S_{0,4}) \]
It follows from Birman--Hilden \cite[Theorem 5]{birm:bihi} that this map is surjective and has kernel $\set{\pm 1}$.  The \emph{point push map} induces an isomorphism (see for instance \cite[Theorem 4.5]{FM})
\[ \Push\colon \pi_{0,3} \longrightarrow \PMod(S_{0,4}). \]
Moreover, $\SL(2,\Z)$ acts on our chosen model of $S_{1,4}$ by linear automorphisms, and this action is compatible with the inclusions in (\ref{Diagram1}).  This gives a map $\PMod(S_{1,1}) \ra \PMod(S_{1,4})$. Let $\Gamma(2)$ denote the principal congruence subgroup of level two in $\SL(2,\Z)$, $\Gamma$ be the image of $\Gamma(2)$ in $\PMod(S_{1,4})$, and write $\tau \in \Gamma$ for the class of multiplication by $-1$.

\begin{lem}\label{EquivariantLemma}
There is a surjection $\Gamma(2) \to \pi_{0,3}$ with kernel $\set{\pm 1}$ such that the diagram
\begin{equation}\label{EquivariantDiagram}
\xymatrix{ \Gamma(2) \ar[r] \ar[d] & \pi_{0,3} \ar[d] \\ \Gamma \ar[r] & \Gamma/\innp{\tau}}
\end{equation}
commutes, with the vertical arrows isomorphisms.
\end{lem}

\begin{proof}  This follows immediately from the description above.  The Birman--Hilden map $i$ restricts to a surjection
\[  i\colon \Gamma(2) \ra \Mod(S_{0,4}) \]
with kernel $\set{\pm 1}$; composing with the inverse of the point-pushing isomorphism yields the upper horizontal map in the diagram.  The left-hand vertical map comes from the action of $\SL(2,\Z)$ on our model of $S_{1,4}$, and the right-hand vertical map from the fact that $-1 \in \Gamma(2)$ maps to $\tau \in \Gamma$.
\end{proof}

\noindent We briefly explain how Lemma \ref{EquivariantLemma} can be expressed in the language of moduli spaces (better, moduli stacks) of algebraic curves.   Write $\MM$ for the moduli space of data $(C,P_1,P_2,P_3,P_4)$, where $C$ is a smooth complex genus $1$ curve and the $P_i$ are distinct points on $C$ such that $P_i - P_j$ is a $2$--torsion divisor on the Jacobian of $C$.  There is a natural isomorphism from $\MM$ to $X(2)$, the moduli stack of elliptic curves with level two structure, which sends $(C,P_1,P_2,P_3,P_4)$ to the elliptic curve $(C,P_1)$ with $P_2 - P_1$ and $P_3 - P_1$ as a basis of $2$--torsion.  On the other hand, given $(C,P_1,P_2,P_3,P_4)$ there is a unique nontrivial involution $\iota$ of $C$ fixing all the $P_i$, which affords a morphism
\[ \phi\colon C \longrightarrow C/\iota \]
whose target is a genus $0$ curve.  This morphism induces a morphism $\MM \ra \MM_{0,4}$ defined by
\[ (C,P_1,P_2,P_3,P_4) \mapsto (\phi(C),\phi(P_1),\phi(P_2),\phi(P_3),\phi(P_4)). \]
This map is an isomorphism on coarse moduli spaces, but the target is generically a scheme, while $\MM$ has a generic inertia group of $\Z/2\Z$. The (analytic) fundamental group of $\MM_{0,4}$ is precisely $\Mod(S_{0,4}) \cong \pi_{0,3}$.
Thus, the diagram
\[ \pi_1(X(2)) \widetilde{\longleftarrow} \pi_1(\MM) \longrightarrow \pi_1(\MM_{0,4}) \]
can be written as
\[ \Gamma(2) \widetilde{\longleftarrow} \pi_1(\MM) \longrightarrow \pi_{0,3}. \]
The group $\Gamma$ appearing in the proof of Lemma \ref{EquivariantLemma} is just $\pi_1(\MM)$, and the inclusion of $\Gamma$ in $\Mod(S_{1,4})$ is induced by the inclusion of $\MM$ as the hyperelliptic locus in $\MM_{1,4}$. \smallskip \smallskip


\section{Proof of main theorem}

\noindent We begin by explaining how the problem at hand, which involves finite-index subgroups of $\pi_{1,1}$, can be translated to a problem about finite-index subgroups of $\pi_{0,4}$.

\begin{prop}  Let $H$ be a finite-index subgroup of $\pi_{0,4}$ satisfying the following properties:
\begin{itemize}
\item[(a)] Let $\gamma$ be an element of $\Gamma(2) / \set{\pm 1}$, considered as an outer automorphism of $\pi_{0,4}$, and $\alpha$ an automorphism of $\pi_{0,4}$ lying over $\gamma$.  If $\alpha(H)$ is not conjugate to $H$, then $[H:H \cap \alpha(H)] > 2$.
\item[(b)] $H$ is not contained in $\pi_{1,4}$.
\item[(c)] Write $H_0$ for the intersection $H \cap \pi_{1,4}$. The permutations on the set $\pi_{0,4} / H$ induced by the four puncture classes represent four distinct conjugacy classes in the symmetric group on $N := [\pi_{0,4}:H] = [\pi_{1,4}:H_0]$ letters.
\end{itemize}
Then
\[ \Stab_{\PMod(S_{1,1})}([H_0]_{\pi_{1,1}}) = i^{-1}(\Stab_{\Mod(S_{0,4})}([H]_{\pi_{0,4}})) \]
where
\[  i\colon \PMod(S_{1,1}) \ra \PMod(S_{0,4}) \]
is the surjection defined in the previous section.
\label{pr:04to11}
\end{prop}

\begin{proof}
For $\gamma \in \PMod(S_{1,1})$, our aim is to show that $\gamma \in \Stab_{\PMod(S_{1,1})}([H_0]_{\pi_{1,1}})$ if and only if $i(\gamma) \in \Stab_{\Mod(S_{0,4})}([H]_{\pi_{0,4}})$. We begin by showing the inclusion
\[  i(\Stab_{\PMod(S_{1,1})}([H_0]_{\pi_{1,1}}))  \subset (\Stab_{\Mod(S_{0,4})} [H]_{\pi_{0,4}}). \]
Suppose $\gamma$ preserves the conjugacy class of $H_0$ in $\pi_{1,1}$.  Then we can choose an automorphism $\beta$ of $\pi_{1,1}$ lying over $\gamma$ such that $\beta(H_0) = H_0$.  Alternately, the homomorphism $\pi_{1,4} \ra \textrm{Sym}(N)$ induced by the action on cosets of $H_0$ is unchanged by composition on the left with $\beta$.  By hypothesis, the four puncture classes in $\pi_{1,4}$, which are naturally identified with the four puncture classes in $\pi_{0,4}$, map to distinct conjugacy classes in $\textrm{Sym}(N)$.  Thus, $\beta$ must preserve each of these four puncture classes.  Recall that the action of $\PMod(S_{1,1})$ on $S_{1,4}$  is that of $\SL(2,\Z)$ on $(\R^2 - \Z^2) / 2\Z^2$, and the four puncture classes in $\pi_{1,4}$ are naturally identified with $(\Z/2\Z)^2$.  In particular, the fact that $\beta$ preserves all four puncture classes implies that $\gamma$ lies in $\Gamma(2)$.\smallskip\smallskip

\nid Now let $\alpha$ be an automorphism of $\pi_{0,4}$ lying over $i(\gamma)$. The restrictions of $\alpha$ and $\beta$ to $\pi_{1,4}$ may differ, but by composing $\alpha$ with an inner automorphism of $\pi_{0,4}$ and $\beta$ with an inner automorphism of $\pi_{1,1}$ we can arrange that $\alpha$ and $\beta$ act identically on $\pi_{1,4}$. Now $\alpha(H) \cap H$ contains $H_0$, and thus has index at most $2$ in $H$, which by hypothesis implies that $\alpha(H)$ and $H$ are conjugate in $\pi_{0,4}$.  Thus $\alpha$ lies in $\Mod(S_{0,4})$ and preserves $[H]_{\pi_{0,4}}$, which was the assertion to be shown.\smallskip\smallskip

\nid What remains is to prove that if $\delta \in \Stab_{\Mod(S_{0,4})}([H]_{\pi_{0,4}})$, then any lift of $\delta$ to $\PMod(S_{1,1})$ preserves the conjugacy class of $H_0$ in $\pi_{1,1}$. Let $\alpha$ be an automorphism of $\pi_{0,4}$ lying over $\delta$.  Then $\alpha(H_0)$ and $H_0$ are conjugate by some element $\eta$ of $\pi_{0,4}$. We claim that $\alpha(H_0) \in [H_0]_{\pi_{1,4}}$. If $\eta \in \pi_{1,4}$, then $\alpha(H_0) \in [H_0]_{\pi_{1,4}}$ as claimed. Thus, we may assume that $\eta$ is not in $\pi_{1,4}$. By assumption, $H$ is not contained in $\pi_{1,4}$ and so contains an element $\lambda$ in $\pi_{0,4}$ but not $\pi_{1,4}$. Since $H_0$ is normal in $H$, we have
\[ \eta \lambda H_0 \lambda^{-1} \eta^{-1} = \eta H_0 \eta^{-1} = \alpha(H_0) \]
However $\eta \lambda$ lies in $\pi_{1,4}$ and so $\alpha(H_0)$ and $H_0$ are in fact conjugate in $\pi_{1,4}$. Now let $\tilde{\delta}$ be a lift of $\delta$ to $\PMod(S_{1,1})$.  By the compatibility in Lemma~\ref{EquivariantLemma}, we can choose an automorphism $\beta$ of $\pi_{1,1}$ in the outer class $\tilde{\delta}$, whose action on $\pi_{1,4}$ differs from that of $\alpha$ by an inner automorphism of $\pi_{0,4}$.  In particular, we have that $\beta(H_0)$ is conjugate in $\pi_{1,4}$ to either $H_0$ or $\tau(H_0)$; but these two subgroups are themselves conjugate in $\pi_{1,4}$, as shown in the paragraph above.  Therefore, $\beta(H_0)$ and $H_0$ are conjugate in $\pi_{1,4}$, whence a fortiori in $\pi_{1,1}$. This completes the proof of the proposition.
\end{proof}

\nid We now explain how to construct a subgroup $H$ of $\pi_{0,4}$ satisfying the hypotheses of Proposition~\ref{pr:04to11}.  Let $\Delta$ be a finite-index subgroup of $\Gamma(2)$ containing $\set{\pm 1}$.  Under the identifications in Lemma~\ref{EquivariantLemma}, we can identify $\Delta / \set{\pm 1}$ with a finite-index subgroup of $\pi_{0,3}$.  We denote by $G(\Delta)$ the preimage of $\Delta / \set{\pm 1}$ under the projection $\pi_{0,4} \ra \pi_{0,3}$.  If $z$ is a representative of the puncture class in $\pi_{0,4}$ that vanishes in $\pi_{0,3}$, then $G(\Delta)$ contains $z$ and all its $\pi_{0,4}$--conjugates.  The $\pi_{0,4}$ conjugacy class of $z$ splits up into a finite union of conjugacy classes in $G(\Delta)$; we choose representatives $z_1, \ldots, z_k$ for these classes, and denote the class of $z_i$ by $[z_i]$.

\begin{lem}\label{ActionOnG}
The action of $\Gamma(2)$ on $\pi_{0,4}$ preserves the conjugacy class of $G(\Delta)$.  The permutation action of $\Gamma(2)$ on the conjugacy classes $[z]_i$ is equivalent to its action on the cosets of $\Delta$.
\end{lem}

\begin{proof}
Recall that $\Gamma(2)$ acts on $\pi_{0,4}$ via point-pushing maps; by the Birman exact sequence these act by conjugation on $\pi_{0,3}$.  This proves that $\Gamma(2)$ preserves $G(\Delta)$ up to conjugacy. Now let $F\colon  S \to S_{0,3}$ be the covering space associated to $\Delta/\set{\pm 1} \subset \pi_{0,3}$, and $p$ a point of $S_{0,3}$. The set $F^{-1}(p)$ has order $[\Gamma(2):\Delta]$ and the action of $\pi_{0,3}$ on $F^{-1}(p)$ is isomorphic to the action of $\pi_{0,3}$ on cosets of $\Delta$.  Viewing $S_{0,4}$ as $S_{0,3} \smallsetminus \set{p}$, the $G$--conjugacy classes $[z_i]$ contained in $[z]_{\pi_{0,4}}$ are identified with $F^{-1}(p)$.
\end{proof}

\begin{lem}\label{NormalizerDelta}
Let $\Delta$ be a finite-index subgroup of $\Gamma(2)$ containing $\set{\pm 1}$.  Then there exists a normal subgroup $\Delta' \subset \Delta$ of finite index, containing $\set{\pm 1}$, whose normalizer in $\Gamma(2)$ is $\Delta$.
\end{lem}

\begin{proof}
To begin, note that $\Gamma(2)/\set{\pm 1}$ is a free group of rank 2.  For notational simplicity we denote $\Delta/\set{\pm 1}$ by $\bar{\Delta}$ for the remainder of the proof.  It suffices to construct a normal subgroup of $\bar{\Delta}$ whose normalizer is $\bar{\Delta}$, since the preimage of such a subgroup in $\Delta$ will yield the subgroup promised by the lemma. Let $\gamma_1,\dots,\gamma_k$ be distinct coset representatives for $\bar{\Delta}$ in $\Gamma(2)/\set{\pm 1}$. Select a prime $q$ which is relatively prime to
\[ [\bar{\Delta}: \bar{\Delta} \cap \gamma_j(\bar{\Delta})\gamma_j^{-1}] \]
for all $j$.\smallskip\smallskip

\noindent Now let $c\colon \bar{\Delta} \ra \Z/q\Z$ be a surjective homomorphism, whose kernel we denote $\Delta_c$.  If $\gamma_j$ normalizes $\Delta_c$, we claim it must also normalize $\bar{\Delta}$.  Suppose on the contrary that $\gamma_j$ normalizes $\Delta_c$ but not $\bar{\Delta}$. It must then be that
\[ \Delta_c < \bar{\Delta} \cap \gamma_j\bar{\Delta}\gamma_j^{-1} < \bar{\Delta}. \]
But $[\bar{\Delta}:\Delta_c] = q$, and so $[\bar{\Delta}:\bar{\Delta} \cap \gamma_j \bar{\Delta} \gamma_j^{-1}]$ is either $1$ or $q$.   The latter is impossible by our choice of $q$, so $\gamma_j$ normalizes $\bar{\Delta}$ as claimed.\smallskip\smallskip

\noindent Let $\gamma_{j_1},\dots,\gamma_{j_r}$ be the subset of coset representatives that normalize $\bar{\Delta}$ and note these elements can be viewed as elements in the finite group $\textrm{N}_{\Gamma(2)/\set{\pm 1}}(\bar{\Delta})/(\bar{\Delta}) = \Theta$.  We impose the additional condition on $q$ that it not divide $\abs{\Theta}$.\smallskip\smallskip

\noindent We note first that $\Theta$ acts faithfully on $H^1(\bar{\Delta},\Z)$.  The faithfulness is proved as follows.  First, for any subgroup $\Theta_0<\Theta$, the stabilizer of $\Theta_0$ in $H^1(\bar{\Delta},\Z)$ is precisely $H^1(\innp{\bar{\Delta},\Theta_0},\Z)$, where $\innp{\bar{\Delta},\Theta_0}$ is the group generated by $\bar{\Delta}$ and the elements $\set{\gamma_j}$ corresponding to the elements in $\Theta_0$. Second, the dimension of $H^1(\Lambda,\Z)$ for any $\Lambda<\Gamma(2)/\set{\pm 1}$ is the rank of the free group $\Lambda$. Since the rank of $\bar{\Delta}$ is strictly greater than the rank of $\innp{\bar{\Delta},\Theta_0}$ for all non-trivial $\Theta_0<\Theta$, we see that the subspace of $H^1(\bar{\Delta},\Z)$ fixed by any non-trivial subgroup $\Theta_0$ has smaller rank. In particular, removing all of these subspaces from $H^1(\bar{\Delta},\Z)$, we are left with an infinite subset of cohomology classes whose stabilizer in the $\Theta$-action is trivial.  Choosing such a class $v$, it is plain that the image of $v$ in $H^1(\bar{\Delta},\F_q)$ also has trivial stabilizer, for almost all primes $q$.  Choose such a $q$, and let $c$ be the image of $v$ in $H^1(\bar{\Delta},\F_q)$.  Then $c$ is a homomorphism from $\bar{\Delta}$ to $\Z/q\Z$.\smallskip\smallskip

\noindent Suppose $\gamma_{i}$ normalizes $\Delta_c$.  Then by the first part of the argument, $\gamma_i$ must normalize $\bar{\Delta}$, and so in particular it must lie in $\Theta$.  But by our choice of $c$, an element of $\Theta$ which preserves $\Delta_c$ must be trivial.  We conclude that no element of $\Gamma(2)/\pm 1$ outside $\bar{\Delta}$ normalizes $\Delta_c$.\smallskip\smallskip

\noindent Finally, the desired subgroup $\Delta'$ can be taken to be the pullback of $\Delta_c$ to $\Gamma(2)$; by construction, $\Delta'$ contains $\set{\pm 1}$, is normal in $\Delta$, and has $\Delta$ as its normalizer in $\Gamma(2)$.
\end{proof}

\nid Let $\Delta'$ be a normal subgroup of $\Delta$ as provided by Lemma~\ref{NormalizerDelta}, and write $G$ for $G(\Delta')$.  Now $G$ is the fundamental group of a punctured surface, and in particular is free.  For each prime $p$, let
\[ \phi_p\colon  G \ra H^1(G,\B{F}_p) \]
be the canonical epimorphism. Let $\ell,p_1,p_2,p_3$ be four distinct large odd primes. To be precise, let $\mathcal{P}_{bad}$ is the set of primes dividing either the index of $G$ in $\pi_{0,4}$, or some index
\[  [G: G \cap \nu G \nu^{-1}] \]
as  $\nu G \nu^{-1}$ ranges over the finite set of conjugates of $G$ in $\pi_{0,4}$.  We want $p_1, p_2, p_3, \ell$ to be larger than any element of $\mathcal{P}_{bad}$ and will mean by "large" precisely this. Note that the set of large primes is co-finite. For a subset $A$ of $H^1(G,\B{F}_p)$, the $\B{F}_p$--span of $A$ will be denoted by $\B{F}_p[A]$.\smallskip\smallskip

\nid Define
\[ H_\ell = \phi_\ell^{-1} (\B{F}_\ell[\phi_\ell([z_1])]) \]
where $[z_1]$ is one of the conjugacy classes in $G(\Delta)$ making up $[z]_{\pi_{0,4}}$; in general it will be a finite union of conjugacy classes in $G$.\smallskip\smallskip

\nid Let $[\gamma_1], [\gamma_2], [\gamma_3] \in \pi_{0,4}$ be the three puncture classes which do not vanish in $\pi_{0,3}$.  For each $i=1,2,3$, denote by $Y_i$ the intersection of $[\gamma_i]$ with $G$; this is a finite union of conjugacy classes in $G$.  Define
\[ H_i = \phi_{p_i} ^{-1} (\B{F}_{p_i}[\phi_{p_i}(Y_i)]). \]
Finally, we define $H = H_\ell \cap H_1 \cap H_2 \cap H_3$.  (We emphasize that the use of $H_1, H_2,$ and $H_3$ is simply to ensure that the third condition of Proposition~\ref{pr:04to11} is satisfied.)

\begin{prop}  The stabilizer of $[H]_{\pi_{0,4}}$ in $\Mod(S_{0,4})$ is $\Delta / \set{\pm 1}$.
\label{pr:stabilizer}
\end{prop}

\begin{proof}  For simplicity, write $\bar{\Delta}$ for $\Delta / \set{\pm 1}$. First, we observe that $\bar{\Delta}$ preserves the conjugacy class of $H$.  By Lemma~\ref{ActionOnG}, we know $\bar{\Delta}$ preserves the conjugacy class of $G$, and the sets $[z_1], Y_1, Y_2, Y_3$; it follows that $\bar{\Delta}$ preserves the conjugacy class of $H$, which is defined in terms of these.  Now suppose $\gamma \in \Mod(S_{0,4})$ preserves the conjugacy class of $H$, but is not in $\bar{\Delta}$.  Then $\gamma([z_1]) = [z_j]$ for some $j \neq 1$.  Define
\[ H_\ell^j = \phi_\ell^{-1} (\B{F}_\ell[\phi_\ell([z_j])]) \]
and
\[ H^j = H_\ell^j \cap H_1 \cap H_2 \cap H_3. \]
Then $\gamma(H) = H^j$ and what we need to prove is that $H^j$ and $H$ are not conjugate in $\pi_{0,4}$.  Suppose on the contrary that
\[ H^j = \eta H \eta^{-1} \]
for some $\eta \in \pi_{0,4}$.  Then we have
\[ H_j \subset G \cap \eta G \eta^{-1}  \subset G \]
But the index $[G:H_j]$ is a product of powers of $p_1, p_2, p_3, \ell$, while the index $[G: G \cap \eta G \eta^{-1}]$ is coprime to $p_1, p_2, p_3, \ell$ by hypothesis; so the latter index must be $1$, which is to say that $\eta G \eta^{-1} = G$.  In particular, this implies that the projection of $\eta$ to $\pi_{0,3}$ normalizes $\bar{\Delta}'$.  By our choice of $\Delta'$, this implies that $\eta$ projects to $\bar{\Delta}$, so that $\eta$ itself lies in $G(\Delta)$. However, we see that $\eta [z_1] \eta^{-1} = [z_1]$, contradicting the hypothesis that $j \neq 1$.
\end{proof}

\begin{prop} $H$ satisfies the conditions of Proposition~\ref{pr:04to11}.
\end{prop}

\begin{proof} First of all, $G$ contains $[z]_{\pi_{0,4}}$, and thus is not contained in $\pi_{1,4}$.  Since $H$ has odd index in $G$, we have that $H$ is also not contained in $\pi_{1,4}$.\smallskip\smallskip

\nid Suppose $\gamma \in \Mod(S_{0,4})$ does not preserve the conjugacy class of $H$; that is, it does not lie in $\bar{\Delta}$.  Let $\alpha$ be an automorphism of $\pi_{0,4}$ in the outer class of $\gamma$.  Then $\alpha(H) \cap H$ is a subgroup of $H$ containing $H \cap H^j$ for some $j \neq 1$.  Since $[H: H \cap H^j]$ is a power of $\ell$, we certainly have that $[H: \alpha(H) \cap H] > 2$ once it is not $1$.\smallskip\smallskip

\nid Finally, we consider the action of the puncture classes $\gamma_1, \gamma_2, \gamma_3$ on the cosets of $H$ in $\pi_{0,4}$. The {\em order} of the permutation attached to $\gamma_i$ is just the smallest integer $n_i$ such that every $\pi_{0,4}$--conjugate of $\gamma_i^{n_i}$ is contained in $H$.  It follows from the definitions that $n_1$ is prime to $p_1$ but a multiple of $p_2$ and $p_3$ , and that the analogous statements hold for $n_2,n_3$.  Similarly, the order of the permutation given by the action of $z$ on $\pi_{0,4} / H$ is a multiple of $p_1 p_2 p_3$.  So no two of these four permutations have the same order, and a fortiori they are non-conjugate in $\textrm{Sym}(N)$.
\end{proof}

\nid It now follows from Proposition~\ref{pr:04to11} and Proposition~\ref{pr:stabilizer} that the stabilizer of $[H]_{\pi_{1,1}}$ in $\Gamma(2)$ is precisely $\Delta$.  So $\Delta$ is a Veech group, as advertised.

\section{Pushing curves from $\MM_{g,n}$ down to $\MM_g$}

\nid We conclude by proving, as promised in the introduction, that the Teichm\"{u}ller curves we construct in $\MM_{g,[n]}$ can in fact be thought of as Teichm\"{u}ller curves in moduli spaces of unmarked algebraic curves.

\begin{prop} \label{pr:mg}
Let $X/\overline{\B{Q}}$ be an algebraic curve.  Then there exists an integer $g'$ and a Teichm\"{u}ller curve $V$ in $\MM_{g'}(\C)$ birational to $X_\B{C}$.
\end{prop}

\begin{proof}  As in the main argument, write $X$ as a Belyi cover of $\P^1$ and let $\Delta$ be the corresponding finite index subgroup of $\Gamma(2)$.  By Theorem \ref{th:veech} we may choose a finite index subgroup $\Delta_0$ of $\pi_{1,1}$ such that $\Delta$ is the stabilizer of the conjugacy class $[\Delta_0]$.\smallskip\smallskip

\nid First, choose some large prime $p$ and let $\Delta_p$ be the kernel of the projection from $\Delta_0$ to $H^1(\Delta_0,\Z/p\Z)$. Visibly, $\Delta$ stabilizes $[\Delta_p]$ and we assert for sufficiently large $p$ that $\Delta$ is precisely the stabilizer of $[\Delta_p]$. To begin, note that for all $p$, the group $\Delta_0/\Delta_p$ is an elementary $p$--subgroup of $\Nor_{\pi_{1,1}}(\Delta_p)/\Delta_p$. By the Sylow Theorems, $\Delta_0/\Delta_p$ is contained in a Sylow $p$--subgroup $P$. Notice that if $\Delta_0/\Delta_p \ne P$, then
\[ [\pi_{1,1}:\Delta_0] \geq [\Nor_{\pi_{1,1}}(\Delta_p):\Delta_0] \geq [P:\Delta_0/\Delta_p] \geq p, \]
which is impossible for sufficiently large primes $p$. Therefore, for large primes $p$, we now see that $\Delta_0/\Delta_p$ is a Sylow $p$--subgroup of $\Nor_{\pi_{1,1}}(\Delta_p)/\Delta_p$. Next, assume that there exists Sylow $p$--subgroup $P$ that is distinct from $\Delta_0/\Delta_p$. It follows then that
\[ [P: P \cap \Delta_0/\Delta_p] \geq p. \]
On the other hand, by elementary group theory, we know that
\[ [\pi_{1,1}:\Delta_0] \geq [\Nor_{\pi_{1,1}}(\Delta_p):\Delta_0] \geq [\Nor_{\pi_{1,1}}(\Delta_p)/\Delta_p: \Delta_0/\Delta_p] \geq [P: P \cap \Delta_0/\Delta_p]. \]
Hence, for sufficiently large $p$, we see that $\Delta_0/\Delta_p$ is a normal Sylow $p$--subgroup of $\Nor_{\pi_{1,1}}(\Delta_p)/\Delta_p$. In particular, any element that stabilizes $[\Delta_p]$ must stabilize $[\Delta_0]$, and so $\Delta$ must be the stabilizer of $[\Delta_p]$ as asserted. This argument, in turn, gives us another expression of $X$ as a Teichm\"{u}ller curve.  We have a diagram
\[ X \longrightarrow \MM_{g,[n]} \stackrel{f}{\longrightarrow} \MM_{g',[n']}, \]
where the first map is the birational embedding associated to our original choice of $\Delta_0$, and the second map $f$ is defined as follows. Let $C$ be a genus $g$ curve and $S$ a set of $n$ points on $G$. Set $C'$ be the maximal Galois cover of $C$ unramified outside $S$ whose Galois group is an elementary abelian $p$--group, and $S'$ be the preimage of $S$ in $C'$. Finally, we define $f(C,S)$ to be $(C',S')$, and the composite map $X \ra \MM_{g',[n']}$ is the realization of $X$ as Teichm\"{u}ller curve attached to $\Delta_p$.\smallskip\smallskip

\nid It remains to show that the composite
\[ X \longrightarrow \MM_{g',[n']} \longrightarrow \MM_{g'}, \]
where the second map is the the projection forgetting the marked points, is generically one-to-one. We proceed via contradiction and set $A = (\Z/p\Z)^{2g + n - 1}$ be the Galois group of any covering $C'/C$ as above. By assumption, there exists a curve $Y$ of genus $g'$ which is expressible in two different ways as a $A$--cover of a genus $g$ curve.  In particular, $\Aut(Y)$ contains two distinct copies of $A$, and so
\[ \abs{\Aut(Y)} \geq p\abs{A}. \]
On the other hand, it follows from the Riemann--Hurwitz formula that
\[ g' <  c\abs{A}. \]
for some constant $c=c(g,n)$ depending only on $g$ and $n$.  In particular, since an algebraic curve of genus $g'$ can have at most $84(g'-1)$ automorphisms, we have a contradiction once $p$ is larger than $84c$. Thus, if two points of $X$ map to the same point of $\MM_{g'}$, they map to the same point of $\MM_{g,[n]}$; and since the map from $X$ to $\MM_{g,[n]}$ is generically injective, so is the map from $X$ to $\MM_{g'}$.  We have thus expressed a curve birational to $X$ as a Teichm\"{u}ller curve in a moduli space of closed Riemann surfaces, as desired.
\end{proof}


\providecommand{\bysame}{\leavevmode\hbox to3em{\hrulefill}\thinspace}


\noindent Department of Mathematics \\
University of Wisconsin at Madison \\
Madison, WI 53706, USA \\
email: {\tt ellenber@math.wisc.edu}\\

\noindent Department of Mathematics \\
University of Chicago \\
Chicago, IL 60637, USA \\
email: {\tt dmcreyn@math.uchicago.edu}\\


\end{document}